\documentclass[12pt]{article}

\usepackage{geometry}
\usepackage[utf8]{inputenc}
\usepackage[english]{babel}
\usepackage[T1]{fontenc}
\usepackage{hyperref}
\usepackage{amsmath}
\usepackage{amssymb}
\usepackage{amscd}
\usepackage{amsthm}
\usepackage{amsfonts}
\usepackage{amscd}
\usepackage{graphicx}
\usepackage[all,cmtip]{xy}
\usepackage{cite}
\let\Im\undefined

\DeclareMathOperator\Crit {Crit}
\renewcommand\ker {\mathop{\mathrm{Ker}}}
\DeclareMathOperator\Exc {Exc}
\DeclareMathOperator\Hom {Hom}
\DeclareMathOperator\Sym {Sym}
\DeclareMathOperator\Frac {Frac}
\DeclareMathOperator\Arc {Arc}
\DeclareMathOperator\Gal {Gal}
\DeclareMathOperator\Sp {Spec\,}

\DeclareMathOperator\Im {Im\,}

\DeclareMathOperator\Div {Div}
\DeclareMathOperator\Tor {Tor}

\newcommand{\tens}{\otimes}

\newcommand{\al}{\alpha}

\newcommand{\Om}{\Omega}

\newcommand{\inj}{\hookrightarrow}

\newcommand{\la}{\lambda}
\newcommand{\ga}{\gamma}
\newcommand{\ov}{\overline}
\newcommand{\La}{\Lambda}
\newcommand{\Ga}{\Gamma}

\newcommand{\wh}{\widehat}
\newcommand{\wt}{\widetilde}

\newcommand{\inv}{^{-1}}
\newcommand{\cart}{\ar@{}[rd]|-\square}

\newcommand{\N}{\mathbb{N}}
\newcommand{\Z}{\mathbb{Z}}

\newcommand{\n}{{(n)}}
\newcommand{\sh}{{\mathrm{sh}}}

\newcommand{\sep}{{\mathrm{sep}}}
\newcommand{\up}[1]{``{#1}$''$}
\renewcommand{\Rsh}{R^\sh}
\newcommand{\Rshc}{\wh{R^\sh}}

\newcommand{\mb}[1]{{\mathbb{#1}}}
\newcommand{\mc}[1]{{\mathcal{#1}}}
\newcommand{\mf}[1]{{\mathfrak{#1}}}

\author{Cyrille Corpet\footnotemark[2]}
\title{Around the Mordell-Lang and Manin-Mumford conjectures in positive characteristic}

\begin{document}

\maketitle
\footnotetext[2]{Doctorant contractuel, Institut de Mathématiques, \'Equipe \'Emile Picard, Université Paul Sabatier, 118 Route de Narbonne, 31062 Toulouse Cedex 9, France, E\nobreakdash-mail:~cyrille.corpet@math.univ-toulouse.fr}
\theoremstyle{plain}
\newtheorem{thm}{Theorem}[section]
\newtheorem*{thmain}{Theorem \ref{main}}
\newtheorem{lem}[thm]{Lemma}
\newtheorem{prop}[thm]{Proposition}
\newtheorem{cor}[thm]{Corollary}
\newtheorem{conj}[thm]{Conjecture}
\newtheorem*{claim}{Claim}

\theoremstyle{definition}
\newtheorem{defi}{Definition}[section]
\newtheorem{exmp}{Example}[section]

\theoremstyle{remark}
\newtheorem*{rems}{Remarks}
\newtheorem*{rem}{Remark}
\newtheorem*{note}{Note}
\newtheorem{case}{Case}

% ------------- La tête de section de l'intro est dans intro.tex, après l'abstract ---------
\begin{quotation}
\subsubsection*{Abstract}
{
\hspace \parindent We give a complete proof for the implication from the Manin-Mumford conjecture to the Mordell-Lang conjecture in positive characteristic, using integral models of semi-abelian varieties over a ring of formal power series, and the machinery of jet schemes.
}
\end{quotation}

\section*{Introduction}
\hspace \parindent The aim of this article is to prove that, in positive characteristic, the Manin-Mumford conjecture implies the Mordell-Lang conjecture. To prove this, we will follow the idea of R\"ossler in \cite{RosA}, which is itself based on the work of Hrushovski \cite{Hru}, Buium \cite{Buium} and many others. More precisely, we will extend the result of \cite{RosA}, which is restricted to abelian varieties over fields of transcendence dimension $1$ and finitely generated subgroups. However this paper will contain the whole proof of the result announced, and it is not required of the reader to have previously read \cite{RosA}. Note that the present paper, together with (for instance) R\"ossler and Pink's proof of the Manin-Mumford conjecture \cite{PinkRo}, gives the first proof of the Mordell-Lang conjecture which is completly algebro-geometric, since previous proofs as Hrushovski's \cite{Hru} used some deep results in model theory.

\bigskip

Let $A$ be a group variety over a field $F$. We say that it is a semi-abelian variety if it fits in an exact sequence of group varieties (which we call a presentation of the semi-abelian variety)
\[0\to T\to A\to A_0\to 0\]
where $A_0$ is an abelian variety and $T$ is a torus, \emph{ie} a group scheme over $F$ such that $T\tens_F\ov F\simeq \mb G_{m,\ov F}^d$, for some $d\geq 0$. Here, $\mb G_{m,\ov F}$ is the multiplicative group of $\ov F$.

Suppose from now on that $F$ is algebraically closed. If we have an irreducible closed subvariety $X\inj A$, we can define its translation stabilizer $\mathrm{Stab}(X)$ as the closed subgroup variety of $A$ whose points stabilize $X$ in $A$ (see \cite[Exp. VIII, 6]{SGA3b} for a proper definition of this as a variety). We call $C=\mathrm{Stab}(X)^{\mathrm{red}}$.

\begin{defi}
We say that $X\inj A$ is isotrivial up to quotient if $F$ is of positive characteristic, and we have the following data:
\begin{itemize}
 \item a semi-abelian variety $B$ over $F$;
 \item a homomorphism with finite kernel $h:B\to A/C$;
 \item a model $\mb B$ of $B$ over a finite field $\mb F_q\subset F$;
 \item an irreducible closed subscheme $\mb X$ of $\mb B$;
 \item a point $a\in (A/C)(F)$, such that $X/C=a+h_*(\mb X\times_{\mb F_q} F)$.
\end{itemize}

\end{defi}

\begin{rem}
 If for any semi-abelian variety $\mb B$ over a finite field there is no non-trivial isogeny $B\to A/C$, to say that $X$ is isotrivial up to quotient means that it is the translate of a subgroup variety of $A$.
\end{rem}

If $A$ is a semi-abelian variety over an algebraically closed field $F$, we write $\Tor(A(F))$ for the group of torsion points, and if $\Ga\subset A(F)$ is a subgroup, we define its prime-to-p divisor group $\Ga'=\Div^p(\Ga)$
\[\Div^p (\Ga) := \{x \in A(F) \vert\ \exists m \in \mb N,\ p\nmid m\ \mathrm{and}\ [m]_{\mc A} x \in \Ga \}.\]
If $X$ is a subvariety, and $Q\in A(F)$, we write $X^{+Q}$ for the translated subvariety. More generally, if $\mc A\to S$ is a group $S$-scheme, $\mc X\to S$ a subscheme and $Q\in \mc A(S)$, the schematic image of $\mc X$ by translation by $Q$ is another $S$-subscheme of $\mc A$ which we note $\mc X^{+Q}$.

\bigskip

Here are the statements of the Manin-Mumford and Mordell-Lang conjectures in positive characteristic:
\begin{thm}[Manin-Mumford conjecture {[}Scanlon, 2001\cite{Scan}, R\"ossler-Pink, 2004 \cite{PinkRo}{]}]\label{MM}
 Let $A$ be a semi-abelian variety over an algebraically closed field $F$ of characteristic $p>0$, $X$
be an irreducible reduced closed subscheme of $A$.

If $X\cap\Tor(A(F))$ is dense in $X$, then $X$ is isotrivial up to quotient.
\end{thm}

\begin{thm}[Mordell-Lang conjecture {[}Hrushovski, 1996 \cite{Hru}{]}]\label{ML}
 Let $A$ and $X$ be as above, and $\Ga\subset A(F)$ be a finitely generated subroup.

 If $X\cap\Ga'$ is dense in $X$, then $X$ is isotrivial up to quotient.
\end{thm}

\begin{note}
Those two theorems stated in this form have been proved quite recently, but work on the matter both in positive and $0$ characteristics have been developped since the statement of the Mordell conjecture in 1922. For instance, McQuillan \cite{McQML} proved the "full Mordell-Lang conjecture" in characteristic $0$ (see \cite{Maz2000} for a good survey of all this).
\end{note}

The main result of the present article is the following:
\begin{thmain}
 Let $A$ be a semi-abelian variety over a field $L$ which is a finite extension of $\ov k((t))$ for some finite field $k$. Let $X$ be a subvariety and $\Ga\subset A(L)$ a finitely generated subgroup. 

Suppose that for any field extension $L/L'$ and any point $Q\in A_{L'}(L')$, the set $X_{L'}^{+Q}\cap\Tor(A_{L'}(L'))$ is not Zariski dense in $X_{L'}^{+Q}$. Then $X\cap \Ga'$ is not dense in $X$.
\end{thmain}

In the first section, we will see how to pull-back any semi-abelian variety to one over such a field $L$. Let us see first how we can induce Theorem \ref{ML} from Theorem \ref{MM}. We will use many times the following lemma, which is an invariance property of the Mordell-Lang conjecture:
\begin{lem}
 Let $F'/F$ be an extension of algebraically closed fields. Theorem \ref{ML} holds for given $A,X,\Ga'$ if and only if it holds for $A_{F'},X_{F'},\Ga'_{F'}\subset A_{F'}(F')$.
\end{lem}
\begin{proof}
See \cite[Lemma 3.4]{RosA}.
\end{proof}

 So if we have the hypothesis for Theorem \ref{ML} for given $F,A,X,\Ga$, we also have it for the pull-back of this data over $L$, then Theorem \ref{main} implies that we can apply Theorem \ref{MM} for some $L',A_{L'},X_{L'}^{+Q}$, and therefore also for the pull-back of this data to $\ov{L'}$, so by Theorem \ref{MM}, $X_{\ov{L'}}^{+Q}$ is isotrivial up to quotient in $A_{\ov{L'}}$, which is equivalent by the invariance lemma to $X_{\ov L}$ isotrivial up to quotient in $A_{\ov L}$. Using the invariance lemma once again allows us to have the same result over $F$, therefore Theorem \ref{MM} implies Theorem \ref{ML}.

We will first use fields of formal power series to reduce the problem to fields of transcendance degree $1$. This part goes smoothly as long as we choose carefully our field extensions, in order not to change too much the variety. This will be dealt with in the first section, along with the definition and some properties of semi-abelian schemes.

Then we build Jet schemes over complete rings (in Section 2) in order to get subschemes of $X$ in which all points of $\Ga'\cap X$ must be (if we allow translation for $X$ and taking subgroup of $\Ga'$ of finite index). Finally we prove that at least one of these subschemes cannot be $X$ itself, thus proving that $X\cap \Ga'$ is not dense in $X$. This last result , which will be dealt with in section 3, will be given a different proof than the analog result in \cite[\S 4]{RosA}, to give a different point of view on the matter.

\section{Models of semi-abelian varieties}
\hspace \parindent In this section, we will show that for any semi-abelian variety $A$ over a finitely generated field $F$ of characteristic $p>0$, there is a field extension $K/F$, where $K$ is a complete local field, such that $A$ has a nice model ${\mc A}\to \Sp \mc O_K$. Furthermore, we will see that some additional data in $A$ defined over $F$ extend to $\wt{\mc A}$.

We first need to define what a nice model for $A$ can be, and give some of its properties.

\subsection{Semi-abelian schemes with good reduction}
\begin{defi}
 For a given irreducible base scheme $S$, a semi-abelian $S$-scheme with good reduction is a group $S$-scheme $\mc A\to S$ which fits in an exact sequence of group $S$-schemes
\[0\to \mc T\to \mc A\to \mc A_0\to 0\]
with $\mc A_0$ an abelian scheme, and $\mc T$ a torus scheme (\emph{ie} $X$ is locally isomorphic to $\mb G_m^d$ for the fppf topology; in particular, for $x\in S$, $\mc T\tens \kappa(x)$ is a torus over $\kappa(x)$). 

Let $K$ be the function field of $S$. We say that a semi-abelian variety $A\to\Sp K$ has a model of good reduction over $S$ if it has a presentation $(T,A_0)$ and a model $\mc A\to S$ which is a semi-abelian scheme of presentation $(\mc T,\mc A_0)$ with $\mc T$ and $\mc A_0$ models of $T$ and $A_0$ respectively.
\end{defi}

Unless explicitely said so, all of our semi-abelian schemes will have good reduction. Furthermore, we assume from now on that we are given a discre valuation ring $R$ of characteristic $p>0$ and residue field $k$, and put $S=\Sp R$.
\begin{prop}
 Let $A\to S$ be a semi-abelian scheme, and $n$ an integer. Then the multiplication-by-$n$ map $[n]_A$ is an isogeny of group schemes (\emph{ie} it is a flat and surjective finite morphism of group schemes). Furthermore, if $p$ does not divide $n$, then it is also \'etale.
\end{prop}
\begin{proof}
 For the first part, using \cite[$VI_B$]{SGA3a}, it is enough to prove that $\mc A[n]:=\ker[n]_A$ is a finite group scheme over $S$. Since we have a presentation $(\mc T,\mc A_0)$ of $\mc A$ over $S$, supposing $\mc T[n]$ is proper hence finite, the kernel of $[n]_A$ is an extension of a finite group scheme (namely $\mc A_0[n]$) by another one, so it must also be finite. But properness for $\mc T[n]$ is the same, by fppf descent, as properness for $\mb G_m^d[n]$ which is an obvious calculation.
 To prove \'etaleness when $n$ is prime to $p$, it is enough to prove it for the special fiber $\wt A\to \Sp k$ (see \cite[$VI_B$, Prop 2.5]{SGA3a}), and even to prove that $\wt A[n]\to \Sp k$ is \'etale. But we have an exact sequence of finite $k$-scheme
 \[0\to \wt T[n]\to \wt A[n]\to \wt A_0[n]\to 0;\]
 since for $n$ prime to $p$, $\wt T[n]$ and $\wt A_0[n]$ are \'etale (because multiplication by $n$ is \'etale on tori and abelian varieties), so is $\wt A[n]$ (see \cite[IV,\S4,(4.45)]{Moon}).
\end{proof}

\begin{cor}\label{zeta}
 In the previous setting, if $R$ is strictly henselian, $\Div^p(A(S))$ and $A(S)$ are equal as subsets of $A(K^\sep)$, so that there is a map
 \[\zeta:\Div^p(A(S))\to A(S)\]
such that every $x\in\Div^p(A(S))$ seen as a morphism $x:\Sp K^\sep\to A$ extend to $\zeta(x):S\to A$.
\end{cor}
\begin{proof}
 Let $x\in \Div^p(A(S))$, $n$ (prime to $p$) and $P$ such that $[n]x$ is the generic point of $P\in A(S)$. This setting implies that $x$ (seen as a morphism $K^\sep\to A$) factors through $\Sp K^\sep\to[n]^*\wt P\inj A$, where $\wt P$ is the schematic image of $P$ in $A$, which is isomorphic to $S$. Since $S$ is the spectrum of a strictly henselian ring and $[n]$ is \'etale, $[n]^*\wt P$ is a product of (a finite number of) copies of $S$, so the image of the only point in $\Sp K^\sep$ is necessarily the generic point of one of these copies which is an $S$-subscheme of $A$ isomorphic to $S$, \emph{ie} an $S$-section $\zeta(x)\in A(S)$. It is now clear from the construction that $\zeta$ has the desired property.
\end{proof}

\subsection{Integral model of semi-abelian varieties}
\hspace \parindent First, we will need to show that any semi-abelian variety over a field of characteristic $p$ can be defined as a variety over a field of formal power series in one variable $K$, with a model of good reduction over its ring of integers. The next proposition is a close adaptation of \cite[Proposition 6]{RosNot}:

\begin{prop}\label{epsilon}
 Let $A$ be a semi-abelian variety over a field $F$ that is a finitely generated extension of $\mb F_p$. Then there is a finite extension $K$ of $\mb F_p((t))$, such that we have a morphism $F\to K$ and the semi-abelian variety $A_K$ has a model of good reduction over $\mc O_K$, where $\mc O_K$ is the valuation ring of $K$. 

 Furthermore, for a given irreducible closed subvariety $X\inj A$ and given points $\ga_1,\dots,\ga_k\in A(F)$, those can be extended to a subscheme $\mc X\inj \mc A$ and integral points $\wt\ga_1,\dots,\wt \ga_k \in\mc A(\mc O_K)$.
\end{prop}

\begin{proof}
 Assuming $F/\mb F_p$ is transcendental (if it's algebraic, the result is obvious), we can find a smooth curve $U$ over $\mb F_p$ such that $F/K(U):=K_0$ is finitly generated. We choose $U=\Sp \mb F_p[x]\simeq \mb A^1_{\mb F_p}$ with $x\in F$ transcendent over $\mb F_p$, so that $U(\mb F_p)\neq\varnothing$. Using a transcendence basis, we have a finite map $\Sp F\to\Sp K_0\left(X_1,\dots,X_d\right)$, where $d$ is the transcendence degree of $F$ over $K_0$. We define $f:V\to \mb A^d_U$ to be the normalisation of $\mb A^d_U$ in $F$. This scheme is integral, normal, and has $F$ as field of rational functions. Also notice that this morphism is finite and surjective \cite[Proposition 1.1]{MilEC}. In this setting, there is an open $B\subset V$ such that $A\to F$ has a model $\mc A_B$ over $B$ (which is not necessarily a group scheme). By restricting $B$, we can also assume the following properties:
\begin{itemize}
 \item $\mc A_B$ is a group scheme with a presentation 
 \[0\to\mc T_B\to \mc A_B\to \mc A_{0,B}\to 0,\]
 where $\mc T_B$ and $\mc A_{0,B}$ are group schemes model of $T$ and $A_0$ respectively over $B$, for some presentation $(T,A_0)$ of $A$ (all the group axioms are satisfied on an open subset if they are satisfied on the generic point of $B$);
 \item $\mc A_{0,B}$ is a model of good reduction for $A_0$ (an abelian variety has good reduction on a non empty open subset);
 \item $\mc T_B$ is a torus $B$-scheme;
 \item $X$ is the generic fiber of a flat irreducible closed subscheme $\mc X_B\inj\mc A_B$ ;
 \item for $i=1,\dots,k$, $\ga_i$ extends to a non-vanishing section in $\mc A_B(B)$;
 \item $f(B):=W$ is open and $f\inv(W)=B$ (this is a rather technical condition).
\end{itemize}

 Thus we have the following cartesian square:
\[\xymatrix{
\Sp F					\ar[r] \ar[d] \cart	&	B 	\ar[d]\\
\Sp K_0\left(X_0,\dots,X_d\right)	\ar[r]		&	W.}
\]
 Now, since $\mb A^d_U(K_0)$ is dense in $\mb A^d_U$ and $W$ is a non-empty open subset of $\mb A^d_U$, $W(K_0)$ is not empty. From this, we can say that for almost all closed points $u\in U$, the fiber $W_u$ is also not empty. Pick such a point $u$, defined over $\mb F_p$, and $P\in W_u(\mb F_p)$ (which is not empty because W is a non-empty open subset of the affine space): using the decomposition $\mb A^d_U=\mb A^d_{\mb F_p}\times_{\mb F_p} U$, we get $(a_1,\dots,a_d)$ as $\mb F_p$-coordinates of the first coordinate of $P$.

 Let $\wh U$ be the completion of the scheme $U$ along the closed subscheme $u$. We define a morphism $\wh U\to \mb A^d_{\mb F_p}$ as follow: using the fact that $\wh U\simeq \Sp \mb F_p[[t]]$, to give a morphism of schemes $\wh U\to A^d_{\mb F_p}$ is the same as giving a morphism of $\mb F_p$-algebras $\mb F_p[X_1,\dots,X_d]\to \mb F_p[[t]]$. Choosing $(x_1,\dots,x_d)$ algebraically independents elements of $\mb F_p[[t]]$ as images of the $X_i$'s makes this morphism injective, and the corresponding scheme morphism dominant (we can, because $\mb F_p((t))/\mb F_p$ has infinite transcendence degree, and if some $x_i$ is not in $\mb F_p[[t]]$, then we can replace it by its inverse without changing the independence property). Choosing the $x_i$'s such as $x_i=a_i$ mod $t$ makes the morphism of schemes send the closed point to $(a_1,\dots,a_d)$ (if some $x_i$ doesn't satisfy this property, we can translate it by a scalar in $\mb F_p$ to obtain it without interfering with the independence or the integral properties). From this morphism $\wh U\to \mb A^d_{\mb F_p}$ and the natural morphism $\wh U\to U$, we get a morphism $e:\wh U\to\mb A^d_U$, which sends the generic point of $\wh U$ to the generic point of $\mb A^d_U$, and sends the closed point of $\wh U$ to $P\in W_u(\mb F_p)$. Hence $e\inv(W)=\wh U$, and we can draw another cartesian square:
\[\xymatrix{
W 		\ar[d] \cart	&	\wh U\simeq\Sp \mb F_p[[t]]	\ar[l] \ar@{=}[d]\\
\mb A^d_U 		 	& 	\wh U \ar[l]}
\]

We can now pull back $B$ to a scheme $B_1\to \wh U$, as in the following cartesian square:
\[\xymatrix{
B	\ar[d] \cart	&	B_1 \ar[l] \ar[d]\\
W			&	\wh U, \ar[l]}
\]
so $B_1\to \wh U$ is also finite and surjective.

Since $B_1$ dominates $\wh U$, there is a reduced irreducible component $B_1'$ of $B_1$ which also dominates it. Since the morphism $B_1'\to\wh U$ is also finite, it correspond to an extension of integral rings. So if we take $K$ to be the function field of $B_1'$, which is a finite extension of $\mb F_p((t))$ via the isomorphism $\wh U\simeq \Sp \mb F_p[[t]]$, we get that $B_1'\simeq\Sp R$ with $R\subset\mc O_K$ a ring of integers in $K$, where $\mc O_K$ is the integral closure of $\mb F_p[[t]]$ in $K$. This gives us a morphism $\Sp \mc O_K\to B_1'\to B_1$. Hence we have the following commutative diagram, where the arrows not already defined are the obvious ones.

\[\xymatrix{
\Sp F \ar[r] \ar@{=>}[d] \cart & B \ar@{=>}[d] \cart & B_1 \ar[l] \ar@{=>}[d] & \Sp \mc O_K \ar@{-->}[l] \ar@{=>}@/_/[ld]& \Sp K \ar[l] \ar@{=>}[d]\\
\Sp K_0\left(X_0,\dots,X_d\right) \ar[r] & W \ar[d] \cart & \wh U \ar[l] \ar@{=}[d] & & \Sp \mb F_p((t)) \ar[ll]\\
 & \mb A^d_U & \wh U \ar[l] & & }
\]

\noindent\emph{$\xymatrix{\ar@{=>}[r] &}$: finite and dominant morphism}\\
\emph{$\xymatrix{\ar[r] &}$: dominant morphism}\\
\emph{$\xymatrix{\ar@{-->}[r] &} $: morphism}\\

Now the composition $\Sp K\to B$ sends the only point of $\Sp K$ to the generic point of $B$, since it is true for $\Sp K\to W$ and $B\to W$ is dominant and finite. So this gives a field extension $K/F$, because $F$ is the field of rational functions on $B$. From the semi-abelian scheme $\mc A_B/B$, we get another one $\wt{\mc A} \to \Sp \mc O_K$ by pulling back along the morphism in the diagram, which has $A_K$ as generic fiber. The subscheme $\mc X$ and the points $\wt \ga_1\dots\wt\ga_k$ come from the corresponding elements over $B$.
\end{proof}

\begin{rem}
 To make sure that $K$-rational points can be extended to $S$-sections at all times (in the manner of Corollary \ref{zeta}), we could use N\'eron  models for semi-abelian varieties, which exist (see \cite[Chap. 10]{BLRNeron}), but only as scheme locally of finite type over $S$, which brings some unwanted confusion in the rest of the proof.
\end{rem}

From now on, we will assume that $k$ is a finite field of characteristic $p$, $\ov k$ its algebraic closure, and $K$ a finite extension of $k((t))$ of residue field $k$ (this does not loose any generality, since we can always replace $k$ by a finite extension of it). We will use the following notations: $R=\mc O_K$, $\Rshc$ the completion of the strict henselisation of $R$ which is a local ring with maximal ideal $\mf m$, $\Rsh_n=\Rshc/\mf m^{n+1}$ and $\wt S,S,S_n$ the corresponding affine schemes. Also, we let $L$ be the fraction field of $\Rshc$ and $\ov L$ its algebraic closure.

For clarity, we will always use the written forms $\wt{\mc W},\mc W,W_n,W,W_{\ov L}$ respectively for a scheme over $\wt S$ and its pull-backs over $S,S_n,\Sp L$ and $\Sp \ov L$, unless explicitly said so.

\bigskip

In this setting and according to Proposition \ref{epsilon}, any semi-abelian variety $A_{\mathrm{init}}$ in characteristic $p$ has a model of good reduction $\wt{\mc A}\to\wt S$ for some $K$ as above (this is true if $A_{\mathrm{init}}$ is defined over a field finitely generated over $\mb F_p$ but any semi-abelian variety, being a scheme of finite type, has a model over such a field).

\section{Jet schemes over a complete ring}
\hspace \parindent We will define jet schemes in our setting, which is somewhat different from \cite{RosA}, since our base scheme is complete, hence not of finite type over the base field.

\bigskip

Let $R,\Rsh,\Rshc, L, S$ be as defined in the previous section, so that $L=\Frac(\Rshc)$ is a finite extension of $\ov k((t))$ and $S=\Sp \Rshc$. We denote $S\wh\times S$ the scheme $\Sp(\Rsh\wh\tens \Rsh)$, \emph{ie} the spectrum of the completion of $\Rsh\tens_{\ov k} \Rsh$ with respect to the ideal $\mf m\tens \Rsh + \Rsh \tens \mf m$, where $\mf m$ is the maximal ideal in $\Rsh$. There is a natural epimorphism $\Rsh\wh\tens \Rsh\to\Rshc$ which is the completion of the multiplication morphism in $\Rsh$; it induces a closed (diagonal) immersion $S\inj S\wh\times S$.

For $n$ a positive integer, we define $S_\n$ as the $n$-th neighborhood of $S$ seen as the diagonal subscheme $S\inj S\wh\times S$, so if $I$ is the definition ideal of this closed immmersion, 
\[S_\n:=\Sp R_\n:=\Sp R\wh\tens R/I^{n+1}.\]
 It is naturally equipped with the two projection 
 \[\pi_1^\n,\pi_2^\n:S_\n\subset S\wh\times S\rightrightarrows S.\]
 We consider this scheme as an $S$-scheme with $\pi_1^\n$ as structural morphism.

\begin{claim}\label{claim}
The immersion $S\to S\wh\times S$ is regular. More specifically, $I^n/I^{n+1}\simeq \Sym_{\Rshc}^n(I/I^2)$ when seen as a $\Rshc$ module via the isomorphism $\Rsh\wh\tens\Rsh/I\simeq\Rshc$.
\end{claim}
\begin{proof}
Since $L$ is a finite extension of $\ov k((t))$, $\Rshc$ is a (finite) free $\ov k[[t]]$-algebra (because $\ov k[[t]]$ is complete, $\Rshc$ is finitely generated, and because it is torsion free and $\ov k[[t]]$ is a PID, it is free). In order to give a clearer proof, we will assume that $\Rshc=\ov k[[t]]$ (the general case is proved likewise). With this assumption, $\Rsh\wh\tens\Rsh=\ov k[[t,t']]$, together with the two maps $\Rshc\to\Rsh\wh\tens\Rsh : t\mapsto t$ and $t\mapsto t'$.

We can see that $I=(t-t')$, so $I/I^2\simeq (t-t')\cdot \ov k[[t]]$ is a free $\Rshc$-module, and the claim follows.
\end{proof}

Remember that, for a flat and finite  morphism $Y\to X$, and a $Y$-scheme $Z$, the Weil restriction functor $T/X \mapsto \mathrm{Hom}_Y(T\times_X Y,Z)$ is representable by an $X$-scheme, $\mf R_{Y/X}(Z)$, the Weil restriction of $Z$.

\begin{lem}
For $n\geq 0$, the $S$-scheme $S_\n$ is flat and finite.
\end{lem}
\begin{proof}
Since we are dealing with affine schemes, we may work on the corresponding algebras. Assuming once again that $\Rshc=\ov k[[t]]$, we see that $\ov k[[t,t']]/I^{n+1}=\bigoplus_{k\leq n}(t-t')^k\cdot k[[t]]$ so the morphism is finite. 

For flatness, notice that for $n> 0$, we have exact sequences of $\Rshc$-modules
\[ 0\to I^n/I^{n+1}\to R_\n \to R_{(n-1)}\to 0.\]
Since $I^n/I^{n+1}$ is isomorphic as an $R_{(0)}$-module to $\Sym^n_{R_{(0)}}(I/I^2)$, it is a free $\Rshc$ module, so by induction on $n$ ($R_{(0)}\simeq \Rshc$ is obviously flat over $\Rshc$), $R_\n$ is flat over $\Rshc$.
\end{proof}

\begin{defi}
 Let $\mc W$ be an $S$-scheme. The $n$-th jet scheme of $\mc W$ over $S$ is the $S$-scheme 
\[J^n(\mc W):=\mf R_{S_\n/S}(\pi_2^{(n),*}\mc W).\]
\end{defi}

 This construction is covariant in $\mc W$, keeps closed immersions, smooth and \'etale morphisms.

 If $m\leq n$, the closed immersion $S_{(m)}\to S_\n$ induces a surjective morphism
\[\La_{n,m}:J^n(\mc W)\to J^m(\mc W).\]

 We also have canonical maps $\la_n:\mc W(S)\to J^n(\mc W)(S)$ which sends $f:S\to \mc W$ to $J^n(f):J^n(S)=S\to J^n(\mc W)$. As noted in \cite{RosA}, these maps do \emph{not} generally arise from morphisms $W\to J^n(W)$.
 
\begin{lem}
 If $\mc W$ is a smooth $S$-scheme, and $n\geq 1$, then $\La_{n,n-1}$ makes $J^n(\mc W)$ into a $J^{n-1}(\mc W)$-torsor under the vector bundle $\La_{n,0}^*(\Om_{W/S}^\vee)\tens\Sym^{n-1}(I/I^2)$.
\end{lem}

\begin{proof}
 This is a close adaptation of \cite[Lemma 2.3]{RosA}: Let $T\to S$ be an affine $S$-scheme. By definition,
\[J^n(\mc W)(T)=\Hom_{S_{(n)}}(T\times_S S_n,\pi_2^{\n,*}\mc W),\]
and the immersion $S_{(n-1)}\inj S_{(n)}$ induces a maps
\[\Hom_{S_{(n)}}(T\times_S S_\n,\pi_2^{\n,*}\mc W)\to \Hom_{S_{(n-1)}}(T\times_S S_{(n-1)},\pi_2^{(n-1),*}\mc W) \]
which is the functorial description of the morphism $J^n(\mc W)\to J^{n+1}(\mc W)$.

Take $f\in \Hom_{S_{(n-1)}}(T\times_S S_{(n-1)},\pi_2^{(n-1),*}\mc W)$ and look at the fibre above $f$. Viewing $f$ as a $S_\n$-morphism $T\times_S S_{(n-1)}\to \pi_2^{(n),*}\mc W$ via the canonical immersions, the elements of the fibre are precisely the extensions of $f$ to $T\times_S S_\n$. Since the ideal of the immersion $S_{(n-1)}\inj S_\n$ has trivial square, we can apply the theory of infinitesimal extensions of morphisms of smooth schemes (see \cite[Exp. III, Prop 5]{SGA1}) to get that the fibre is an affine space under the group
\[H^0(T\times_S S_{(n-1)},f^*\Om^\vee_{\pi_2^{\n,*}\mc W/S_\n}\tens N)\]
where $N$ is the conormal bundle of the immersion $T\times S_{(n-1)}\inj T\times S_\n$. By flatness of the $S_\n$'s, this is the pull-back of the conormal bundle of $S_{(n-1)}\inj S_\n$, which is $\Sym^{n-1}(I/I^2)$ from the claim. Therefore, we have
\[H^0(T\times_S S_{(n-1)},f^*\Om^\vee_{\pi_2^{\n,*}\mc W/S_\n}\tens N)\simeq H^0(T,f_0^*\Om^\vee_{\mc W/S}\tens \Sym^{n-1}(I/I^2))\]
where $f_0$ is the pull-back of $f$ via $S\to S_\n$, and the lemma is proved.
\end{proof}

\subsection{Jet schemes of commutative group schemes}

\hspace \parindent Let $\mc C$ be a commutative group $S$-scheme. Then the jet schemes $J^n(\mc C)$ are natural group schemes too, and $\La_{n,m}^{\mc C}$ are group scheme morphisms.

\begin{lem}
 Let $n\geq 0$. There is an $S$-morphism $\up {p^n}:\mc C\to J^n(\mc C)$ such that $\La_{n,0}^{\mc C}\circ \up {p^n}=[p^n]_{\mc C}$ and $\up{p^n}\circ \La_{n,0}^{\mc C}=[p^n]_{J^n(\mc C/U)}$.
\end{lem}

\begin{proof}
The kernel of $\La_{n,n-1}$ is $\epsilon^*(\Om^\vee_{\mc C/S})\tens\Sym^{n-1}(I/I^2))$, where $\epsilon$ is the zero section of the group scheme $\mc C$ (See \cite[\S2.2]{RosA}). In particular, it is isomorphic to $\mb G_{a,S}^d$ as a group, with $\mb G_a$ the additive group, so it is killed by the multiplication by $p$ in $J^n(\mc C)$.

Therefore, the kernel of $\La_{n,0}=\La_{n,n-1}\circ\cdots\circ\La_{1,0}$ is killed by $[p^n]_{J^n(\mc C)}$, so for any $\wh R^\sh$-algebra $D$, we have a well defined map: $\mc C(D)\to J^n(\mc C)(D)$ which maps $f\in \mc C(D)$ to $[p^n] \wt f$ where $\wt f$ is any lifting of $f$ in $J^n(\mc C)(D)$. Those maps, being functorial in $D$, give the morphism $\up {p^n}:\mc C\to J^n(\mc C)$, with the desired properties.
\end{proof}

For a semi-abelian scheme $\mc A$, $[p^n]_{\mc A}$ is finite, so $\up{p^n} $ is quasi-finite, and it is also proper, hence finite. Therefore its schematic image, $p^n J^n(\mc A)$, is a closed subscheme of $J^n(\mc A)$ which is finite over $\mc A$.

\begin{lem}
The maps $\La_{m,n}$, $\la_n$ and $\up{p^n}\ $are compatible in the following way:
\begin{itemize}
\item for $n\geq m\geq 0$, $\La_{n,m}\circ \la_n=\la_m$;
\item for a commutative group scheme $\mc C$, $\la_n$ is a homomorphism and, on $S$-points, $[p^n]_{J^n(\mc C)}\circ\la_n=\up{p^n}.$
\end{itemize}
\end{lem}

\begin{proof}
See \cite[Lem. 2.5]{RosA}.
\end{proof}

\subsection{Jet subschemes}

\hspace \parindent Now suppose that we have a closed subscheme $\mc X\inj\mc A$ in a semi-abelian scheme. Then there is also a closed immersion of the jet schemes, and, we can use it to define the critical schemes of $\mc X$:
\[\Crit^n(\mc X):=J^n(\mc X)\cap p^n J^n(\mc A)\subset J^n(\mc A). \]
Furthermore, the morphism $\Crit^n(\mc X)\to\mc X$ is the restriction of $\La_{n,0}$ to $\Crit^n(\mc X)\subset p^n J^n(\mc A)$ so it is finite, and its schematic image is again a closed subscheme which we call the exceptional scheme $\Exc^n(\mc X)\subset \mc X$.

\subsection{Arc schemes}\label{arc}

\hspace \parindent The special fiber of jet schemes gives a positive characteristic analog of the Greenberg functor construction in \cite[II.1.2]{Ray1016}, so we may want to define it independantly as the scheme of arcs: for an $S$-scheme $\mc W$, 
\[\Arc^n(\mc W):=\mathrm{Res}_{S_n/\ov k}(\mc W\times_{S} S_n)=J^n(\mc W)_0.\]

This construction inherits all properties from jet schemes, and in particular, for a commutative group scheme $\mc C$, we have morphisms $\up{p^n}:C_0\to\Arc^n(\mc C)$ which are finite whenever $\mc C$ is a semi-abelian scheme. However, arc shemes are easier to manipulate than jet schemes since they are defined over a ``small'' algebraically closed field $\ov k$ instead of the much larger $S$.

We will also use the fact that, for an integral subscheme of a semi-abelian scheme $\mc X\inj\mc A$,  $\ov k$-points of $\Exc^n_0(\mc X)=\Arc^n(\mc X)\cap p^n\Arc^n(\mc A)$ corresponds to $S_n$ sections in $\mc X(S_n)\cap p^n\mc A(S_n)$.

\section{A density theorem}
\hspace \parindent To prove Theorem \ref{main}, we will first need the following result and its corollaries:

\begin{thm}\label{indefp}
$i)$ For any $n\geq 0$, $\Exc^n_0(\mc X)$ is a closed subscheme of $X_0$ with the following property: 
\[\Exc^n_0(\mc X)(\ov k)=\left\{P\in X_0(\ov k)\vert P\ \text{lifts to an element of}\ \mc X(S_n)\cap p^n\cdot \mc A(S_n)\right\}.\]

$ii)$ There is an integer $m$ such that $\Exc^m_0(\mc X)(\ov k)$ is the set of specializations of indefinitely $p$-divisible points in $\mc X(S)$.

$iii)$ Furthermore, if $\Tor(A_{\ov L}(\ov L))\cap X_{\ov L}$ is not dense in $X_{\ov L}$, $\Exc^m_0(\mc X)$ is a strict subscheme of $X_0$.
\end{thm}

\begin{proof}
The first part of the theorem comes directly from the definition of the exceptional schemes and the last remark of the previous section. 

Since $X_0\to \ov k$ is of finite type, there is an $m$ such that we have a topological equality
\[\bigcap_{n\geq 0} \Exc^n_0(\mc X)=\Exc^m_0(\mc X).\]
Pick a point $P_0\in\Exc^m_0(\mc X)(\ov k)$. This gives us a compatible sequence 
\[\cdots\to\wt {P_n},\to\wt{P_{n-1}},\to\cdots\to{P_0}\]
 where $\wt{P_n}\in \mc X(S_n)\cap p^n \mc A(S_n)$ is a lift-up of $P_0\in\Exc^n_0(\mc X)$ and the arrows are the restriction morphisms $\mc X(S_n)\to\mc X(S_{n-1})$. This sequence defines an element $P\in \mc X(S)=\varprojlim \mc X(S_n)$ (with specialization $P_0$) which is indefinitely $p$-divisible, thus proving $ii)$.
 
Let us show $iii)$ by contraposition, so suppose that every point of $X_0(\ov k)$ lifts up to an indefinitely $p$-divisible point in $\mc X(S)$.

 Since $\mc X\to S$ is flat, $\dim_{L} X=\dim_{\ov k} X_0$, so the set $\Sigma$ of indefinitely $p$-divisible points of $\mc X(S)\subset X(L)$ must be dense in the generic fiber $X$, so the corresponding set $\Sigma_{\ov L}$ of points in $X_{\ov L}$ is also dense.

Let $\tau\in\mathrm{Aut}_K(\ov L)$ be a lifting of the Frobenius automorphism $\varphi\in\Gal(\bar k/k)$ (choose any extension of a lifting to $\Gal(\ov K/K)$). Because $k$ is finite, of cardinality $q$, we get from the Weil conjectures that there is a polynomial $Q\in\Z[X]$ with complex roots of modulus $\sqrt q$ and $q$ (respectively for the abelian part and the torus part), such that for every point $x\in A_0(\ov k)$, $Q(\varphi)(x)=0$. 

We need to define a subset of $A_0(\ov k)$ which maps Frobenius equivariantly onto the indefinitely $p$-divisible points of $\mc A(\ov L)$. To do this, we use the construction of \cite[Proposition 2.3]{RosA'}:

Let $I_p(A_0(\ov k)):=\varprojlim_{l\geq 0} A_0(\ov k)$ where the transition morphisms are all the multiplication-by-$p$ morphism, so $x\in I_p(A_0(\ov k))$ is a sequence $(x_l)_{l\geq 0}\in A_0(\ov k)^{\mb N}$ such that $x_{l+1}=[p]x_l$ for any $l\geq 0$. In other words, $(x_l)_{l\geq 0}$ is a compatible system of indefinitely $p$-divisible points of $A_0(\ov k)$.

Let $n$ be a non-negative integer. From the previous section, we know that there is a map $\up{p^n}:\mc A\to J^n(\mc A)$; if $\mu_n:A_0\to \Arc^n(\mc A)$ is its special fiber, it has the following description on $I_p(A_0(\ov k))$:

from $(x_l)_{l\geq 0}\in I_p(A_0(\ov k))$ we take any lifting $\wt x_n\in \mc A(S_n)$ of $x_n$; we then have $\mu_n((x_l)_{l\geq 0})=[p^n]\wt x_n\in\mc A(S_n)\simeq \Arc^n(\mc A)(\ov k)$ (this description is easily deduced from the way we previously constructed $\up{p^n}\text{).}$

When $n$ varies, those maps are compatible with the restriction morphisms $\mc A(S_n)\to\mc A(S_{n'})$ for $n\geq n'$, so we get a map $\mu_\infty:I_p(A_0(\ov k))\to \varprojlim_{n\geq 0} \mc A(S_n)\simeq\mc A(S)\subset A_{\ov L}(\ov L)$. If we  give $I_p(A_0(\ov k))$ the natural Galois action, it is clear that $\mu_\infty$ is a Frobenius-equivariant morphism, so $\mu_\infty\circ\varphi=\tau\circ\mu_\infty$.

We now need to prove the following lemma, corresponding to \cite[Lemma 2.4]{RosA'} (notice that in this reference, only the abelian case is proved):

\begin{lem}
The image of $\mu_\infty$ is precisely the set of indefinitely $p$-divisible points in $\mc A(S)$:
 \[\Im \mu_\infty=\{P\in \mc A(S)\lvert \forall n\geq 0\ \exists Q\in \mc A(S),\ [p^n]Q=P\}.\]
\end{lem}
\begin{proof}
The inclusion "$\subset$" is obvious, once noticed that all points in $I_p(A_0(\ov k))$ are indefinitely $p$-divisible (a $[p]$-th root of an element is given by a shift of the compatible sequence). To prove "$\supset$", we will first assume the following claim:
\begin{claim}
 For all indefinitely $p$-divisible point $P\in\mc A(S)$, there is a sequence 
 \[\cdots\xrightarrow{[p]} P_n\xrightarrow{[p]}P_{n-1}\xrightarrow{[p]}\cdots\xrightarrow{[p]} P.\]
\end{claim}
Assuming this, from an indefinitely $p$-divisible point $P$, we get a sequence as in the claim whose restriction $x=(x_l)_{l\geq 0}$ to $A_0(\ov k)$ is in $I_p(A_0(\ov k))$. It is then an easy exercise with the description given above to check that $\mu_\infty(x)=P$, using the fact that $\Rshc=\varprojlim_n \Rsh_n$.

Let's now prove the claim. Let $P\in\mc A(S)$ be an indefinitely $p$-divisible point. By definition, there is a sequence $Q_n\in\mc A(S)$ such that $[p^n]Q_n=P$. The set $K_1:=\{[p^{n-1}]Q_n,n\geq 1\}$ is a subset of $[p]\inv(P)$ which is finite, because $\ker [p]$ is a finite and proper $S$-scheme. Therefore, there is $P_1\in K_1$ such that $[p^{n-1}]Q_n=P_1$ for an infinite number of indices, which means that $P_1$ is also indefinitely $p$-divisible. By induction, we get the sequence needed to prove the claim.

\end{proof}

\noindent\emph{End of the proof of Theorem \ref{indefp}.}

From the Frobenius equivariance of $\mu_\infty$, we get that $\Sigma_{\ov L}=\Im \mu_\infty \cap \mc X(S)$ is in the kernel of $Q(\tau)$, so, being also dense in $X_{\ov L}$, we can apply \cite[Theorem 3.5]{PinkRo} to say that $X_{\ov L}$ is isotrivial up to quotient, therefore its torsion points must be dense in it.
\end{proof}

\begin{rems}
 \begin{itemize}
  \item Notice that we have shown the stronger result that if a subvariety of $A$ has a dense set of indefinitely $p$-divisible points, then it is isotrivial up to quotient. Thus, it gives (with the rest of the proof below) a full proof of the Mordell-Lang conjecture, without the need for the Manin-Mumford conjecture.
  \item In the case of abelian varieties, a shorter proof can be provided, using the fact \cite[Theorem 1.3]{RosA'} that indefinitely $p$-divisible points are torsion points, although the proof given there is essentially the same as this one.
  \item This result should also be true in mixed characteristic, when we assume $S=\Sp \wh{\mc O_K^\sh}$ where $K$ is a finite extension of $\mb Q_p$. Indeed, the Greenberg functor \cite{Ray1016} should give an analog of the scheme of arcs good enough to be able to transpose the proof we have given in that setting.
 \end{itemize}

\end{rems}

The following corollary is a uniform version of (a weak form of) the theorem:

\begin{cor}\label{Exc0}
 Suppose that for all $a\in\wt{\mc A}(\wt S)$, $\Tor(A_{\ov L}({\ov L}))\cap X_{{\ov L}}^{+a}$ is not dense in $X_{{\ov L}}^{+a}$. Then there exists an integer $\al$ such that for any $a\in\wt{\mc A}(\wt S)$,
\[U_\al(a):=\Exc^\al_0(\mc X^{+a})(\ov k)=\{P\in X_0^{+a}(\ov k)\vert P\ \text{lifts to an element of}\ \mc X^{+a}(S_\al)\cap p^\al\cdot \mc A(S_\al)\}\]
is not Zariski dense in $X_0^{+a}$.
\end{cor}

\begin{proof}
 For any $a\in\wt{\mc A}(\wt S)$, let $m(a)$ be the smallest integer $m$ such that $U_m(a)$ is not dense in $X_0^{+a}$ (the theorem gives us the existence of such an $m$). Suppose there exists a sequence $(a_n)_{n\in\N}\in\wt{\mc A}(\wt S)^\N$ contradicting the statement, so that $m(a_n)$ strictly increases. Since $\wt{\mc A}$ is an $\wt S$-scheme of finite type, $\wt{\mc A}(\wt S)$ is compact for the topology induced by the one on $R$, the underlying ring of $\wt S$. This comes from the fact that $R$ itself is compact, being a discrete valuation ring with finite residue field. Therefore we may assume (by taking a subsequence if necessary) that $(a_n)$ converges for this topology towards $a\in\wt{\mc A}(\wt S)$. Taking yet another subsequence, we may also assume that $a_n$ and $a$ have the same image in $\wt{\mc A}(\wt S_n)$, since this space is discrete. Then necessarily, $m(a_n)\geq n$ (because it is a strictly increasing sequence), which implies that $m(a)\geq n$ (because $U_n(c)$ depends only on the image of $c$ in $\mc A(S_n)$). However, this cannot be true for all $n\in \mb N$ since $m(a)$ is finite, therefore our supposition was wrong and the corollary is proved.
\end{proof}

\hspace \parindent We can now finish the proof of the main theorem:

\begin{thm}\label{main}
 Suppose that for any field extension $L'/L$ and any point $Q\in A(L')$, the set $X_{L'}^{+Q}\cap\Tor(A(L'))$ is not Zariski dense in $X_L^{+Q}$. Then $X\cap\mathrm{Div}^p(\Ga)$ is not Zariski dense in $X$.
\end{thm}

\begin{proof}

We need the following lemma:

\begin{lem}\label{al}
 Under the hypothesis of the theorem, there is a positive integer $\al$, depending only on $\mc A$ and $\mc X$, such that $\Exc^\al(\mc X^{+Q})\subsetneq\mc X^{+Q}$ for all $Q\in\Ga$.
\end{lem}

\begin{proof}
 The hypothesis of the theorem allows to apply Corollary \ref{Exc0}, since $\wt{\mc A}(\wt S)\subset A(L)$. Therefore there is an integer $\al$ such that $\Exc^\al_0(\mc X^{+Q})$ is a proper subscheme of $\mc X_0^{+Q}$ for any $Q\in\Ga\subset\wt{\mc A}(\wt S)$, but, since $\mc X^{+Q}$ is integral, its closed subscheme $\Exc^\al(\mc X^{+Q})$ is distinct from $\mc X^{+Q}$ because it is true of the special fibers $\Exc^\al_0(\mc X^{+Q})\inj X^{+Q}_0$.
\end{proof}

 Let's prove the theorem: Let $Q\in \Ga$, $n\in\mb N$; using the $\zeta$ function from Corollary \ref{zeta}, we have the following:
\begin{align*} 
   \mc X^{+Q}(S)\cap p^n\cdot \zeta(\Ga') & = \La_{n,0}\circ \la_n[\mc X^{+Q}(S)\cap p^n\cdot \zeta(\Ga')] \\
     & \subset \La_{n,0}[\la_n(\mc X^{+Q}(S))\cap\la_n(p^n\cdot \zeta(\Ga'))] \\
     & \subset \La_{n,0}[J^n(\mc X^{+Q})\cap p^n\cdot J^n(\mc A)(S)]\subset \La_{n,0}[\rm {Crit}^n(\mc X^{+Q})]\\
     & = \Exc^n(\mc X^{+Q}).
  \end{align*}
 
 Let $\al$ be as in the lemma; choosing a set of representatives $Q_1,\cdots,Q_k\in\Ga$ for $\Ga'/p^\al\Ga'=\Ga/p^\al \Ga$ (which is assumed to be finite), this computation shows that 
 \[\mc X(S) \cap \left(Q_i+p^\al\cdot\zeta(\Ga')\right)\subset \rm {Exc}^\al(\mc X^{-Q_i})^{+Q_i}\quad \mathrm{so}\ 
  \mc X(S)\cap \zeta(\Ga') \subset\bigcup_{i\leq k} \rm {Exc}^\al(\mc X^{-Q_i})^{+Q_i}.\]
 According to the lemma, each of the $\rm {Exc}^\al(\mc X^{+Q_i})$'s is non dense, so it has (relative) dimension strictly smaller than the one of $\mc X$ (since $\mc X$ is irreducible). Thus a finite union of their translates cannot be dense in $\mc X$.
% suppose $X\cap \Ga'$ is dense in $X$. Since $\Ga'/p\Ga'$ is finite and $X$ is irreducible, for some $P_1\in\Ga'$, $(X^{+P_1})\cap p\cdot \Ga'$ is dense in $X^{+P_1}$. We may assume that $P_1\in\Ga$, because the natural map $\Ga\to\Ga'$ induces an isomorphism $\Ga/p\Ga\simeq\Ga'/p\Ga'$. Iterating this construction, which is allowed since $p^i\Ga'/p^{i+1}\Ga'$ is finite for all $i\geq 0$, gives a sequence of points $P_i\in\Ga$ such that $X^{+P_1+\cdots+P_n}\cap p^n\cdot\Ga'$ is dense in $X^{+P_1+\cdots+P_n}$ for all $n\geq 1$. Note that since $P_1+\cdots+P_n\in\Ga\subset A(S)$, $\mc X^{+P_1+\cdots+P_n}$ is well defined.

% Remember that $\Ga'=\Div^p(\Ga)\subset \Div^p(\mc A(S))$, and $S$ is strictly henselian, so Corollary \ref{zeta} applies and there is a map $\zeta:\Div^p(\mc A(S))\to \mc A(S)$ which is a right inverse to generization.
%Let $\al$ be as in the lemma. For $Q=\sum_{i=1}^\al P_i\in\Ga$, $X^{+Q}\cap p^\al\cdot\Ga'$ is dense in $\mc X^{+Q}$. However, $X^{+Q}\cap p^\al\cdot\Ga'\subset \mc X^{+Q}(S)\cap p^\al\cdot\zeta(\Ga')$ as topological spaces, and this gives the following computation:

%so this exceptional subscheme must be $\mc X^{+Q}$, which contradicts Lemma \ref{al}.

\end{proof}

\begin{rem}
The proof of Theorem \ref{main} indicates that it should be possible to give an effective bound on the number of irreducible components of the Zariski closure of $X\cap \Ga'$, at least in the abelian case, using some intersection theory, as has been done for characteristic $0$ in \cite{BuiEff}. However, as seen in \cite{BuiVol}, the bound in positive characteristic should be much smaller, under some mild assumptions on a field of definition for $X$. 
\end{rem}

%\bibliography{necessaire}{}
%\bibliographystyle{alphanum}

\end{document}